\documentclass[1pt]{article}

\usepackage[T1]{fontenc}
\usepackage[utf8]{inputenc}
\usepackage{amsmath,amssymb,amsthm}
\usepackage{bm}
\usepackage{enumitem}
\usepackage{needspace}
\usepackage{xcolor}
\usepackage{geometry}

\IfFileExists{newtxtext.sty}{%
  \IfFileExists{newtxmath.sty}{%
    \usepackage{newtxtext}
    \usepackage{newtxmath}
  }{\usepackage{mathptmx}}%
}{\usepackage{mathptmx}}

\setlist{itemsep=-0.7pt}

\theoremstyle{plain}
\newtheorem{theorem}{Theorem}[section]
\newtheorem{proposition}[theorem]{Proposition}
\newtheorem{corollary}[theorem]{Corollary}
\newtheorem{lemma}[theorem]{Lemma}
\newtheorem*{maintheorem}{Main theorem}

\theoremstyle{definition}

\newtheorem{example}[theorem]{Example}
\newtheorem{remark}[theorem]{Remark}

\def\R{\mathbb{R}}

\def\eps{\varepsilon}
\def\la{\langle}
\def\ra{\rangle}
\def\leq{\leqslant}
\def\le{\leq}
\def\geq{\geqslant}
\def\ge{\geq}

\newcommand{\x}{\bm{x}}

\newcommand{\0}{\bm{0}}     

\begin{document}

\title{A Spectral Approach to Fixed Points in Product Spaces}

\author{
  Aleksander \'Cwiszewski\\
  \small Faculty of Mathematics and Computer Science, Nicholas Copernicus University, Toru\'n, Poland
  \and
  Wojciech Kryszewski\\
  \small Institute of Mathematics, Technical University of Lodz, {\L}\'od\'z, Poland
}

\date{}
\maketitle

\begin{abstract}
We study fixed points of mappings on finite products of Banach spaces whose coordinatewise Lipschitz estimates are controlled by a nonnegative matrix $A$. Our main result states that, if every coordinate space has normal structure and $r(A)\leq1$, then every such self-mapping of a nonempty weakly compact convex set has a fixed point. The result is a direct generalization of the
Browder--G\"ohde--Kirk principle. The proof combines classical nonexpansive fixed point theory with Frobenius normal form, passing successively to coordinate projections and invariant fibres. As preparation, the classical reduction of Perov's theorem to the Banach contraction principle through adapted lattice norms is revisited. Applications include a Dirichlet boundary value problem for systems of abstract second order differential equations and a periodic problem for ODE systems satisfying one-sided matrix estimates.
\end{abstract}


\section{Introduction}

Fixed point problems arising from systems of equations
frequently admit componentwise estimates in which the
dependence between different components is described by
a nonnegative matrix. To illustrate this situation, let
\[
X:=\prod_{i=1}^N X_i,
\]
where $(X_i,\|\cdot\|_i)$, $i=1,\ldots,N$, are real or
complex Banach spaces, and let
\[
f=(f_1,\ldots,f_N):X\to X
\]
be Fr\'echet differentiable. Suppose that its partial
derivatives are uniformly bounded, that is,
\[
a_{ij}:=\sup_{x\in X}
\|\partial_j f_i(x)\|_{\mathcal L(X_j,X_i)}
<\infty,
\qquad i,j=1,\ldots,N.
\]
Then the mean value inequality yields
\[
\|f_i(x)-f_i(y)\|_i
\leq \sum_{j=1}^N a_{ij}\|x_j-y_j\|_j,
\qquad x,y\in X,\quad i=1,\ldots,N.
\]
Thus, the coefficient $a_{ij}$ provides an upper bound
for the influence of the $j$-th variable on the $i$-th
component of $f$.

More generally, consider a mapping $T=(T_1,\ldots,T_N)$
satisfying an estimate of the form
\[
\|T_i(x)-T_i(y)\|_i
\leq \sum_{j=1}^N a_{ij}\|x_j-y_j\|_j,
\qquad i=1,\ldots,N,
\]
where $A=[a_{ij}]$ is a nonnegative matrix.
Such inequalities arise naturally in the study of
coupled nonlinear equations and provide a natural framework for fixed point results.

When the spectral radius $r(A)$ of $A$ is strictly less
than one, the classical contraction principle of
Perov~\cite{Perov} ensures the existence and uniqueness of a fixed
point in the appropriate complete setting. The purpose of this paper is to establish a fixed
point theorem under the condition
$r(A)=1$. In this case the situation becomes substantially different. The usual contraction argument
is then no longer available in general, and additional
geometric assumptions are needed to guarantee fixed
point existence. Our result applies to arbitrary weakly
compact convex invariant subsets of finite products
of Banach spaces with normal structure.

Recall that a Banach space $(Y,\|\cdot\|)$ has
\emph{normal structure} if every nonempty bounded
closed convex subset $D\subset Y$ with positive
diameter contains a point $z\in D$ such that
\[
\sup_{y\in D}\|z-y\|<\operatorname{diam}(D),
\]
where
\[
\operatorname{diam}(D)
:=\sup_{x,y\in D}\|x-y\|.
\]
In particular, every uniformly convex Banach
space has normal structure
(see \cite[Chap.~6]{Kirk-Goebel}).

Our main result can now be stated as follows.

\begin{maintheorem}[Theorem~\ref{critical-main}]
Let $N\geq1$, and let $(X_i,\|\cdot\|_i)$,
$i=1,\ldots,N$, be real or complex Banach spaces
with normal structure. Let
\[
K\subset X_1\times\cdots\times X_N
\]
be nonempty, weakly compact and convex, and let
$T=(T_1,\ldots,T_N):K\to K$.
Suppose that there exists a nonnegative matrix
$A=[a_{ij}]_{i,j=1}^N$ such that
\[
\|T_i(x)-T_i(y)\|_i
\leq \sum_{j=1}^N a_{ij}\|x_j-y_j\|_j,
\quad x,y\in K,\quad i=1,\ldots,N.
\]
If the spectral radius
\[
r(A):=\max\{|\lambda|:
\lambda\text{ is an eigenvalue of }A\}
\]
satisfies $r(A)\leq1$, then $T$ has a fixed point
in $K$.
\end{maintheorem}

For $N=1$ and $A=[1]$, this theorem reduces to
the classical fixed point theorem for nonexpansive
mappings on weakly compact convex subsets of
Banach spaces with normal structure
(see Theorem~\ref{BGK}),
see Browder~\cite{Browder},
G\"ohde~\cite{Gohde} and Kirk~\cite{Kirk}. If all the coordinate spaces are
uniformly convex, then the weak compactness
assumption on $K$ can be replaced by the requirement
that $K$ be nonempty, bounded, closed and convex.
Indeed, finite products of reflexive Banach spaces
are reflexive. The result thus extends the scalar fixed point existence principle to mappings controlled by a nonnegative matrix.

The relation of our result to Perov's fixed point theorem
\cite{Perov} lies in extending the spectral condition
$r(A)<1$ to $r(A)\leq1$. Perov's theorem provides the existence of a unique
fixed point and convergence of successive approximations
in an arbitrary complete vector-valued metric space.
At the critical threshold $r(A)=1$, our theorem
preserves the existence conclusion under additional
geometric assumptions, namely weak compactness,
convexity and normal structure. However, neither uniqueness nor convergence of successive
approximations is guaranteed when $r(A)=1$.

It is worth observing that the Perov theorem may be
derived from the classical Banach contraction principle,
as shown, in particular, by Jachymski and
Klima~\cite{Jachymski} and Cvetkovi\'c~\cite{Cvet};
a certain argument can also be found in
Khamsi~\cite{Khamsi}.

The weighted lattice norm approach employed in our
analysis has its roots in the classical theory of
nonnegative matrices and positive operators; see,
for example, Horn and Johnson~\cite{Horn} and
Berman and Plemmons~\cite{Berman}. Particularly
relevant in this context is the work of Rheinboldt
and Vandergraft~\cite{Van}, where weighted monotone
norms and their relationship with spectral properties
of positive operators are studied in the more general
setting of finite-dimensional ordered spaces.

For the reader's convenience, we provide
self-contained proofs of the auxiliary results
from positive-matrix theory needed in our arguments,
without claiming novelty for these classical tools
or for the strict contraction case.

It should be emphasized that our main theorem applies,
in particular, to matrices of the form
$$A=\begin{pmatrix}1&c\\0&1\end{pmatrix},\qquad c>0,$$
for which
$$r(A)=1,\qquad A^n=\begin{pmatrix}
1&nc\\0&1
\end{pmatrix},\qquad n\geq1.$$
Since the powers of $A$ are unbounded, there is no
norm on $\mathbb R^2$ whose induced operator norm
satisfies $\|A\|\leq1$. Consequently, the given
matrix Lipschitz estimate cannot, in general, be
reduced to a scalar nonexpansive estimate by means
of a single equivalent norm. This illustrates why
the treatment of reducible matrices requires an
argument beyond global scalarization.

The proof proceeds by distinguishing cases of the irreducible
and reducible matrix $A$ and combines Perron--Frobenius theory of nonnegative matrices with geometric fixed point arguments.
For an irreducible
matrix $A$ with spectral radius one, a strictly positive
Perron eigenvector determines a weighted supremum
norm in which the mapping becomes nonexpansive.
The theorem of Belluce--Kirk--Steiner
\cite[Theorem~2.1]{Belluce} implies that the
product space preserves normal structure, allowing
the application of the Browder--G\"ohde--Kirk
fixed point principle. The general reducible case is treated through the Frobenius normal form providing a block triangular decomposition, allowing the fixed point problem to be resolved successively on the corresponding invariant fibres. The argument does not require $K$ to have
a Cartesian product structure.

Matrix estimates and related extensions of Perov's
fixed point theorem have attracted considerable
attention; see, for example,
\cite{Petrusel,Pascale} and the extensive
bibliography in \cite{Cvet}.
Such estimates have also played an important role
in the study of systems of nonlinear equations
and integral operators; see, in particular,
Precup~\cite{Precup1,Precup2}.

We illustrate the applicability of our main result
to Dirichlet boundary value problems for systems
of abstract second-order differential equations
and to periodic problems for abstract differential
equations subject to dissipative matrix-valued
one-sided Lipschitz conditions.
Notably, these applications require neither
compactness nor weak continuity of the nonlinear
operators involved.

The paper is organized as follows.
Section~\ref{strict-case} introduces the notation
and basic definitions, recalls Perov's fixed point
theorem, and develops the lattice-norm techniques
used in the strict spectral case $r(A)<1$.
Section~\ref{recursive} addresses the critical
case $r(A)=1$ and contains the proof of
Theorem~\ref{critical-main}, together with
its consequences and an example illustrating
the sharpness of the spectral condition.
Section~\ref{applications} presents applications
to Dirichlet boundary value problems and
periodic solutions of abstract differential
equations.
Finally, the Appendix collects the classical
results from positive-matrix theory used in
the proofs, including the Perron--Frobenius
theorem and Frobenius normal form, together
with auxiliary material on semi-inner products
in Banach spaces and comparison principles
for differential inequalities.

\section{Notation, preliminaries and Perov's theorem}\label{strict-case}

\subsection{Notation and definitions}\label{basic-notation}

\subsubsection*{Order and matrix notation.} We use the coordinatewise order on $\R^N$: for $u,v\in\R^N$
\begin{equation}\label{order}
u\leq v\ \ \Longleftrightarrow \ \  u_i\leq v_i,
\ \ \mbox{ for each }  i=1,\ldots,N.
\end{equation}
We write $u<v$ if $u\leq v$ and $u\neq v$, and $u\ll v$ if $u_i<v_i$ for every $i$. In particular, $u\gg0$ means that every coordinate is strictly positive. The same order conventions apply entrywise to matrices: $A\geq0$ or $A\in\R_+^{N\times N}$ means that $A$ is nonnegative (all its elements are nonnegative real numbers), and $A\gg 0$  means that $A$ is positive. A vector $u\gg0$ satisfying $Au\leq u$ is called a \emph{positive subinvariant vector}. For a real square matrix $A$, its transpose is denoted by $A^T$, its spectrum by $\sigma(A)$, and its spectral radius by
$$
r(A):=\max\{|\lambda|:\lambda\in\sigma(A)\}.
$$
\indent A matrix is \emph{reducible} if the coordinate subspace $\operatorname{span}\{e_j:j\in J\}$ is invariant under $A$ for some nonempty proper subset $J\subset\{1,\ldots,N\}$; otherwise it is \emph{irreducible}; here $e_j$ are the standard coordinate vectors.

\subsubsection*{Lattice norms and subordinate matrix norms.}
A norm $|\cdot|_*$ on $\R^N$ is a \emph{lattice norm}, also called a \emph{monotone norm}, if
$$
|u|\leq|v| \  \ \Longrightarrow\ \  |u|_*\leq|v|_*,
\ \mbox{ for all } \ \ u,v\in\R^N,
$$
where $|\cdot | :\R^N\to \R_+^N$, with $\R_{+}=[0,+\infty)$, is defined by
$$
|z|=(|z_1|,\ldots,|z_N|), \ \ \mbox{ for any } z=(z_1,\ldots,z_N).
$$
Examples include the usual norms $|\cdot|_p$, $1\leq p\leq\infty$, and the weighted norms
$$
|z|_{\infty,u}:=\max_{1\leq i\leq N}\frac{|z_i|}{u_i},
\qquad
|z|_{1,w}:=\sum_{i=1}^Nw_i|z_i|,
\qquad u,w\gg0.
$$
For a given lattice norm $|\cdot|_*$ on $\R^N$,  the corresponding subordinate matrix norm is
\begin{equation}\label{subo}
|A|_*:=\sup_{z\neq0}\frac{|Az|_*}{|z|_*}.
\end{equation}


\subsubsection*{Matrix Lipschitz conditions.} A \emph{vector-valued metric} on a set $X$ is a function $d:X\times X\to\R_+^N$ such that $d(x,y)=0$ if and only if $x=y$, $d(x,y)=d(y,x)$, and $d(x,z)\leq d(x,y)+d(y,z)$ componentwise. Equivalently, $d=(d_1,\ldots,d_N)$ is a finite separating family of pseudometrics, i.e. each $d_i$ satisfies the metric axioms except possibly separation, and $d_i(x,y)=0$, for every $i = 1,\ldots, N$, implies $x=y$.

Every lattice norm $|\cdot|_*$ on $\R^N$ yields an ordinary metric $d_*: X\times X\to \R_{+}$, given by
\begin{equation}\label{metric*}
d_*(x,y):=\big|d(x,y)\big|_*,\qquad x,y\in X.
\end{equation}

In order to consider convergence, Cauchy sequences and completeness, one may use any of the scalarizations above, since all norms on $\R^N$ are equivalent.

A mapping $T:X\to X$ satisfies a \emph{matrix Lipschitz condition} if, for a nonnegative matrix $A=[a_{ij}]$,
\begin{equation}\label{lipf1}
d_i(Tx,Ty)\leq\sum_{j=1}^N a_{ij}d_j(x,y), \ \ \mbox{ for each } \ \ i=1,\ldots,N.
\end{equation}
Equivalently,
\begin{equation}\label{lipf2}
d(T(x),T(y))\leq A\,d(x,y),\qquad x,y\in X.
\end{equation}

\noindent Observe that the matrix estimate retains separate estimates for the interactions between components; a single scalar Lipschitz constant generally loses this information.

\subsubsection*{Products of Banach spaces.}
Let us consider a Cartesian product
\begin{equation}\label{product}
X:=\prod_{i=1}^N X_i,
\end{equation}
where $(X_i,\|\cdot\|_i)$, $i=1,\ldots, N$, are real or complex Banach spaces over the same scalar field. We use the vector-valued norm and metric
\begin{align}
\|x\|&:=(\|x_1\|_1,\ldots,\|x_N\|_N),\label{vnorm}\\
d(x,y)&:=(\|x_1-y_1\|_1,\ldots,\|x_N-y_N\|_N)
\label{vmetric}
\end{align}
defined for any  $x=(x_1, \ldots, x_N)\in X$ and $y
=(y_1, \ldots, y_N)\in X$.

A lattice norm $|\cdot|_*$ on $\R^N$ gives the norm
\begin{equation}\label{norm*1}
\|x\|_*:=\bigl|(\|x_1\|_1,\ldots,\|x_N\|_N)\bigr|_*.
\end{equation}
For $T=(T_1,\ldots,T_N):K\to K$, $K\subset X$, we can consider the matrix Lipschitz condition \begin{equation}\label{Alip1}
\|T(x)-T(y)\|\leq A\|x-y\|, \ \ \mbox{ for all } \ \ x,y\in K,
\end{equation}
or, componentwise,
\begin{equation}\label{Alip2}
\|T_i(x)-T_i(y)\|_i\leq\sum_{j=1}^N a_{ij}\|x_j-y_j\|_j,
\ \ \mbox{ for all } \ \ i=1,\ldots,N.
\end{equation}
Thus $a_{ij}$ bounds the influence of the $j$-th variable on the $i$-th component.

All product norms of the form \eqref{norm*1} are equivalent. Boundedness, closedness and completeness are therefore independent of this choice. Equivalent norms also have the same continuous linear functionals and the same weak topology, so weak compactness is unchanged. Convexity is an algebraic property and is likewise independent of the norm. In particular, $K\subset X$ is bounded if and only if each coordinate is bounded on $K$.

\subsubsection*{Normal structure and nonexpansiveness.}
For a nonempty bounded set $D$ in a normed space $X$ with the norm $\|\cdot\|_X$, we define its diameter by
$$
\operatorname{diam}(D):=\sup_{x,y\in D}\|x-y\|_X.
$$
A point $z\in D$ is \emph{diametral} if $\sup_{y\in D}\|z-y\|_X=\operatorname{diam}(D)$. A nonempty bounded closed convex set $C$ has \emph{normal structure} if every closed convex subset $D\subset C$ with positive diameter contains a point that is not diametral. A Banach space is said to have normal structure if each of its nonempty bounded closed convex subsets has normal structure. Unlike weak compactness, normal structure need not be preserved by an arbitrary equivalent renorming.
\begin{theorem}[Browder--G\"ohde--Kirk {\cite{Browder,Gohde,Kirk}}]\label{BGK}
Let $Y$ be a Banach space, let $C\subset Y$ be a nonempty convex
set, and let $F:C\to C$ be nonexpansive, i.e.
$$
\|F(x)-F(y)\|_Y\leq\|x-y\|_Y,
\qquad x,y\in C.
$$
Then $F$ has a fixed point in each of the following situations
\begin{enumerate}
\item[\rm(i)]
$Y$ is uniformly convex and $C$ is bounded and closed;

\item[\rm(ii)]
$C$ is weakly compact and has normal structure.
\end{enumerate}
\end{theorem}

\subsection{Scalarization and lattice renorming}

Let $|\cdot|_*$ be a lattice norm in $\R^N$. Since all norms on $\R^N$ are equivalent, completeness of the vector-valued metric space $(X,d)$ implies completeness of $(X,d_*)$ with the metric given by \eqref{metric*}. From \eqref{lipf2} and the lattice property we obtain, for any $x,y\in X$,
$$
d_*\bigl(T(x),T (y)\bigr)=\big|d(T(x),T(y))\big|_*\leq
\big|A\,d(x,y)\big|_*\leq|A|_*\,d_*(x,y),
$$
where $|A|_*$ is the matrix norm defined in \eqref{subo}. Similarly, see \eqref{norm*1}, condition \eqref{Alip1} implies
\begin{equation}\label{lip}
\|T(x)-T(y)\|_*
\leq
|A|_*\,\|x-y\|_*, \ \ \mbox{ for all }\ \ x,y\in K.
\end{equation}
Thus scalarization replaces the matrix Lipschitz estimate by an ordinary Lipschitz estimate. Since
$r(A)\leq |A|_*$ for every subordinate matrix norm, the spectral radius is a lower bound for all Lipschitz constants obtained in that way.

The natural question is therefore whether one can choose a lattice norm so that the subordinate norm of $A$ is as close as desired to $r(A)$. The answer is affirmative.

\begin{proposition}\label{lattice-renorming}
Let $A\in\R^{N\times N}$ be nonnegative. For every $\eps>0$ there exists a lattice norm $|\cdot|_\eps$ in $\R^N$ such that
$$|A|_\eps<r(A)+\eps.$$
\end{proposition}

\begin{proof} Let $\rho=r(A)+\frac{\eps}{2}$ and let $|\cdot|_0$ be any lattice norm on $\R^N$. Define
\begin{equation}\label{extremal-lattice}
|z|_\eps := \sup_{k\geq 0} \frac{\big|A^k|z|\big|_0}{\rho^k},
\qquad z\in\R^N,
\end{equation}
where $|z|=(|z_1|,\ldots,|z_N|)$. Since $r(A)/\rho<1$, the sequence of
$\big|A^k\big|_0/\rho^k$, $k\geq 1$,  is bounded, and hence \eqref{extremal-lattice} is finite. The term $k=0$ gives
$$
|z|_\eps\geq |z|_0,
$$
so $|z|_\eps=0$ implies $z=0$. Absolute homogeneity of $|\cdot |_\eps$ is immediate. Moreover, for any $z_1, z_2\in\R^N$,
$$
|z_1+z_2|\leq |z_1|+|z_2|,
$$
and, since $A^k\geq0$, we have
$$
A^k|z_1+z_2|
\leq
A^k|z_1|+A^k|z_2|.
$$
The lattice property of $|\cdot|_0$ therefore yields
$$
|z_1+z_2|_\eps
\leq
|z_1|_\eps+|z_2|_\eps.
$$
If $|z_1|\le |z_2|$, then, since $A^k\ge0$,
$$
A^k|z_1|\le A^k|z_2|,
$$
and hence
$$
|z_1|_\eps\le |z_2|_\eps.
$$
Thus $|\cdot|_\eps$ is a lattice norm.

\medskip

Finally, since $A\geq0$, for any $z\in\R^N$,
$$
|Az|\leq A|z|,
$$
and hence
\begin{align}\label{seq-ineq-norm}
|Az|_\eps=\sup_{k\geq0}\frac{\big|A^k|Az|\big|_0}{\rho^k}
\leq\sup_{k\geq0}\frac{\big|A^{k+1}|z|\big|_0}{\rho^k}
=\rho\sup_{j\geq1}\frac{\big|A^j|z|\big|_0}{\rho^j}
\leq\rho |z|_\eps.
\end{align}
Thus $|A|_\eps\leq\rho=r(A)+\frac{\eps}{2}<r(A)+\eps.$
\end{proof}
Proposition \ref{lattice-renorming} asserts that
\begin{equation}\label{inf-lattice}
\inf\bigl\{|A|_*:\ |\cdot|_*\text{ is a lattice norm on }\R^N\bigr\}=r(A).
\end{equation}
It corresponds to Goebel~\cite{Goebel} in a nonlinear setting. The infimum in \eqref{inf-lattice} is attained in some special cases.
\begin{proposition}\label{os}
Let $A\in \R^{N\times N}$ be nonnegative with $r(A)>0$.  The following are equivalent
\begin{enumerate}
\item[{\em (i)}] $|A|_*=r(A)$ for some lattice norm $|\cdot |_*$;
\item[{\em (ii)}] $\sup_{k\geq 0} |A^k|_0/r(A)^k<\infty$ for some lattice norm  $|\cdot|_0$;
\item[{\em (iii)}] every eigenvalue $\lambda$ with $|\lambda|=r(A)$ is semisimple.
\end{enumerate}
\end{proposition}
\begin{proof}
The implication (i) $\Rightarrow$ (ii)  follows from $|A^k|_*\le r(A)^k$.  The equivalence (ii) $\Leftrightarrow$ (iii) comes from the Jordan form.
In order to see (ii) $\Rightarrow$ (i) define
$$
|z|_*:=\sup_{k\ge0} \frac{\bigl|A^k|z|\bigr|_0}{r(A)^{k}}.
$$
Then it is clear that, by the assumption in (ii), $|\cdot|_*$ is well-defined. By the arguments from the previous proof, $|\cdot|_*$ is a lattice norm and, by the same arguments like in
\eqref{seq-ineq-norm} we get $|Az|_* \leq r(A)|z|_*$, for all $z\in\R^N$, i.e. $|A|_*\leq r(A)$.
\end{proof}

Therefore, whenever $r(A)<1$, due to Proposition \ref{lattice-renorming}, we can choose a norm $|\cdot|_*$  so that  $|A|_*<1$ and, for $d_*=|d|_*$,
$$
d_*\bigl(T(x),T(y)\bigr)\leq |A|_* \,d_*(x,y), \ \ \mbox{ for any }  \ \ x,y\in X.
$$
Thus Perov's theorem becomes a direct consequence of the Banach contraction principle.
\begin{theorem}[Perov \cite{Perov}]\label{Perov}
Let $(X,d)$ be a nonempty complete vector-valued metric space and let
$T:X\to X$ satisfy \eqref{lipf2} for some nonnegative matrix $A$. If
\begin{equation}\label{sr}
r(A)<1,
\end{equation}
then $T$ has a unique fixed point $x_*\in X$.
\end{theorem}
There is an immediate counterpart for mappings acting on product spaces.
Let $(X,\|\cdot\|)$ be as in \eqref{product}, \eqref{vnorm}, and let $T=(T_1,\ldots,T_N):X\to X$.
\begin{corollary}\label{Perov1}
If $T:X\to X$ satisfies \eqref{Alip1} and $r(A)<1$, then $T$ has a unique fixed point.
\end{corollary}
The standard proof of Perov's theorem follows the successive-approximation argument of the Banach contraction principle. By Theorem~\ref{r(A)<1-app} in the Appendix, the condition $r(A)<1$ is equivalent to $A^n\to 0$ as $n\to\infty$. Thus the powers of $A$ control the propagation of perturbations under successive iterations.

As mentioned in Introduction the relation between Perov's theorem and the Banach contraction principle has already been studied by several authors. Jachymski and Klima~\cite{Jachymski} showed that Perov's theorem is subsumed by the Banach contraction principle in view of the remetrization approach of Goebel~\cite{Goebel}. Cvetkovi\'c~\cite[Theorem~3.6]{Cvet} also established this implication, using a sufficiently high iterate that is contractive in a scalar metric; see also Khamsi~\cite{Khamsi} for related arguments involving the so-called $b$-metric spaces.

\subsection{Perron--Frobenius adapted lattice norms}

The scalarization technique, relying on an appropriate choice of a lattice norm was introduced in numerical analysis several decades ago (see \cite{Szyld, Van}). In particular, weighted monotone norms and their relationship with the spectral properties of positive matrix operators in the setting of finite-dimensional ordered spaces form the core of the important paper of Rheinboldt and Vandergraft~\cite{Van}.

The purpose of introducing weighted lattice norms in the following theorem is to convert the matrix Lipschitz estimate into an ordinary scalar Lipschitz estimate while preserving the product structure. Although most of the results stated below are known and scattered throughout the literature, e.g. \cite{Van}, we give self-contained proofs of the relevant facts needed in what follows in order to prepare for the passage from $r(A)<1$ to the critical case $r(A)=1$.

\begin{theorem}\label{irr}
Let $A\in\R_+^{N\times N}$ and let $\rho>0$.\\[0.4em]
\noindent{\em (i)} The following conditions are equivalent:
\begin{enumerate}
\item[\rm(a)] $r(A)<\rho$;
\item[\rm(b)] there exists $u\in\R^N$, $u\gg0$, such that $Au\ll \rho u$;
\item[\rm(c)] there exists a weighted $\ell^\infty$-type lattice norm
\begin{equation}\label{winfnorm}
|z|_{\infty,u}:=\max_{1\leq i\leq N}\frac{|z_i|}{u_i},\;\; z\in\R^N,
\end{equation}
where $u\in\R^N$ and $u\gg 0$, such that $|A|_{\infty,u}<\rho$;
\item[\rm(d)] there exists $w\in\R^N$, $w\gg0$, such that $A^Tw\ll \rho w$;
\item[\rm(e)] there exists a weighted $\ell^1$-type lattice norm
\begin{equation}\label{wsumnorm}
|z|_{1,w}:=\sum_{i=1}^N w_i|z_i|,\;\; z\in\R^N,
\end{equation}
where $w\in\R^N$ and $w\gg 0$, such that $|A|_{1,w}<\rho$.
\end{enumerate}
\noindent{\em (ii)} If there exist $u,w\in\R^N$, $u,w\gg 0$ such that $Au\leq \rho u$ and $A^Tw\leq\rho w$, then there is a positive diagonal matrix $M\in\R^{N\times N}$ such that
$$|A|_M\leq \rho,$$
where $|A|_M$ is the matrix norm subordinate to the $\ell^2$-type norm $$|z|_M:=\sqrt{z^TMz},\;\; z\in \R^N.$$

\noindent{\em (iii)} If $A$ is irreducible, then there exist $u,w\gg0$ such that
$$
|A|_{\infty,u}=|A|_{1,w}=r(A).
$$
In particular there is a positive diagonal matrix $M\in\R^{N\times N}$ such that
$$|A|_M=r(A).$$

\end{theorem}

\begin{proof}
(i) Assume first that $r(A)<\rho$ and put $B:=\rho^{-1}A$. Then $B\geq0$ and $r(B)<1$. Hence, in view of Theorem \ref{r(A)<1-app},
$$
(I-B)^{-1}=\sum_{k=0}^{\infty}B^k
$$
exists and is nonnegative. Let $ u:=(I-B)^{-1}e$ where $e:=(1,\ldots,1)^T$.
Then $u=Bu+e$ and  $u\geq e\gg0$. We also get
$$
Au=\rho (u-e)\ll \rho u,
$$
which proves the implication (a)\,$\Rightarrow$\,(b).

Now suppose that $u\gg0$ and $Au\ll \rho u$. Set
$$
q:=\max_{1\leq i\leq N}\frac{(Au)_i}{\rho u_i}.
$$
Then $0\leq q<1$ and $Au\leq q \rho\,u$. Observe that, for any $z\in\R^N$,
$$
|z|\leq |z|_{\infty,u}u,
$$
and therefore
$$
|Az|
\leq A|z|
\leq |z|_{\infty,u}Au
\leq q \rho \,|z|_{\infty,u}u.
$$
Hence $|A|_{\infty,u}\leq q \rho<\rho$, which proves (b)\,$\Rightarrow$\,(c). Finally, (c) implies (a), since due to  the general inequality \eqref{subo1}, we get $r(A)\leq |A|_{\infty,u}<\rho$.
Thus conditions (a), (b), and (c) are equivalent.

Applying the same argument to $A^T$ gives (a)\,$\Rightarrow$\,(d). If $w\gg0$ and $A^Tw\ll \rho w$, let
$$
q:=\max_{1\leq i\leq N}\frac{(A^Tw)_i}{\rho w_i}<1.
$$
Then, for any $z\in\R^N$,
\begin{align*}
|Az|_{1,w}
&=
\sum_{i=1}^N w_i|(Az)_i|
\leq
\sum_{i=1}^N w_i(A|z|)_i\\
&=
\sum_{i=1}^N(A^Tw)_i|z_i|
\leq
q \rho \sum_{i=1}^Nw_i|z_i|
=
q \rho \,|z|_{1,w}.
\end{align*}
Thus $|A|_{1,w}\leq q \rho < \rho$, proving (d)\,$\Rightarrow$\,(e), while (e)\,$\Rightarrow$\,(a) again follows from $r(A)\leq|A|_{1,w}$.\\

\indent (ii) Let
$$D_u=\mathrm{diag}(u_1,...,u_N),\;\;D_w=\mathrm{diag}(w_1,...,w_N)$$
and let
$$M=D_u^{-1}D_w=\mathrm{diag}(w_1/u_1,...,w_N/u_N).$$
Observe that $|\cdot|_M$ is indeed a lattice norm in $\R^N$ and
$$|z|^2_M=\sum_{i=1}^N\frac{w_i}{u_i}z_i^2,\;\; z\in\R^N.$$
Now let $z\in\R^N$. By the Cauchy–Schwarz inequality,
\begin{align*} (Az)_i^2 =\left(\sum_{j=1}^N a_{ij}z_j\right)^2\leq \left(\sum_{j=1}^N a_{ij}u_j\right) \left(\sum_{j=1}^N a_{ij}\frac{z_j^2}{u_j}\right)\leq \rho u_i\sum_{j=1}^N a_{ij}\frac{z_j^2}{u_j}. \end{align*}
Multiplying by $w_i/u_i$ and summing over $i$, gives
\begin{align*}|Az|_M^2&=(Az)^T M(Az)=\sum_{i=1}^N\frac{w_i}{u_i}(Az)_i^2\leq \rho\sum_{i,j=1}^N w_i a_{ij}\frac{z_j^2}{u_j}\\&=\rho\sum_{j=1}^N \frac{z_j^2}{u_j}\sum_{i=1}^N a_{ij}w_i\leq \rho^2\sum_{j=1}^N\frac{w_j}{u_j}z_j^2=\rho^2z^TMz=\rho^2|z|^2_M.
\end{align*}
Since this holds for every $z\in\R^N$, we get $|A|_M\leq \rho$.\\

\indent (iii) Suppose that $A$ is irreducible. In view of the Perron--Frobenius theorem -- see Theorem \ref{Per-Frob}, there exist right and left Perron vectors $u,w\gg0$ satisfying
$$
Au=r(A)u,
\qquad
A^Tw=r(A)w.
$$
Therefore, for any $\eps>0$,
$$Au\ll (r(A)+\eps)u,\;\; A^Tw\ll (r(A)+\eps)w.$$
By (i) this implies that $|A|_{\infty,u},|A|_{1,w}<r(A)+\eps$, i.e.
$$
|A|_{\infty,u}\leq r(A),
\qquad
|A|_{1,w}\leq r(A).
$$
In view of part (ii) above we have that $|A|_M\leq r(A)$, too.
Since the spectral radius is bounded above by every subordinate operator norm,
$$
r(A)\leq |A|_{\infty,u},
\;\;
r(A)\leq |A|_{1,w}\;\; \text{and}\;\; r(A)\leq |A|_M,
$$
and consequently $|A|_{\infty,u}=|A|_{1,w}=|A|_M=r(A)$.

\end{proof}

\begin{remark} (i) Clearly Theorem \ref{irr} (i) subsumes Proposition \ref{lattice-renorming}.The latter result was stated
and proved in order to show that scalarization in the context of Perov’s theorem can be carried out without appeal to the Perron–Frobenius theory.

(ii) Theorem \ref{irr} (i) shows that, for nonnegative matrices, apart from \eqref{inf-lattice},
\begin{equation}\label{inf-lattice-1}
\inf\bigl\{|A|_{\infty,u}\mid u\in\R^N,\,u\gg 0\bigr\}=r(A).
\end{equation}
This is a reformulation of the Collatz--Wielandt formula (e.g. \cite[\S8.1]{Horn}, \cite[Chap. 2]{Berman} and \cite[Thms.3,4]{Van}).  In general the infimum need not be attained. It is attained by a weighted norm if $A$ is irreducible, or -- more generally -- if and only if $Au\leq r(A)u$ for some $u\gg 0$, see also Proposition \ref{irr1}.

(iii) Theorem \ref{irr} (ii) and the second assertion in (iii) give a particularly attractive alternative to the weighted supremum or weighted $\ell^1$ norm: a weighted Euclidean norm $|\cdot|_M$ induced by a positive definite diagonal matrix $M$. This norm  is uniformly convex, since it is Hilbertian. If each space $X_i$ is uniformly convex, then so is $\|\cdot\|_M=\big|\|\cdot\|\big|_M$. Similarly if each space $X_i$ is a Hilbert space with scalar product $\la\cdot,\cdot\ra_i$, $i=1,\ldots,N$, then so is $X$ and the norm $\|\cdot\|_M$ is determined by the scalar product
$$\la x,y\ra_M:=\sum_{i=1}^N\frac{w_i}{u_i}\la x_i,y_i\ra_i.$$
\end{remark}

\begin{corollary}\label{perov-lattice-cor}
For $A\in\R_+^{N\times N}$, the following conditions are equivalent
\begin{enumerate}
\item[\rm(a)] $r(A)<1$;
\item[\rm(b)] there exists  $u\gg 0$ such that  $Au\ll u$;
\item[\rm(c)] there exists $w\gg 0$ such that  $A^Tw\ll w$;
\item[\rm(d)] there exists a lattice norm $|\cdot|_*$ such that $|A|_*<1$.
\end{enumerate}
\end{corollary}
\begin{remark} Let $u\in\R^N$ and $u\gg 0$. The condition $Au\leq u$ has a clear geometric meaning. Namely it is equivalent to the invariance $A(Q)\subset Q$, where
$$Q:=\{z\in\R^N\mid |z|\leq u\}.$$
Observe that $Q$ is the unit ball for $|\cdot|_{\infty,u}$. Indeed if $z\in Q$, then $|z|\leq u$, hence $|Az|\leq A|z|\leq Au\leq u$, so $Az\in Q$. Conversely, $u\in Q$, and $A(Q)\subset Q$ implies $Au\in Q$; since $Au\geq 0$, this gives $Au\leq u$.
\end{remark}

We have the following result changing the matrix Lipschitz condition into an ordinary Lipschitz condition
with respect to a suitable weighted norm; it
will play a key role in the critical case $r(A)=1$.

\begin{proposition}\label{irr1}
Let $(X,\|\cdot\|)$ be as in \eqref{product}, \eqref{vnorm}, and let $T=(T_1,\ldots,T_N):K\to K$ with $K\subset X$ satisfy \eqref{Alip1} with $A\in\R_+^{N\times N}$.
If $u\gg 0$, then
$$
\|T(x)-T(y)\|_{\infty,u} \leq |A|_{\infty,u}\,\|x-y\|_{\infty,u},
$$
where  $\|x\|_{\infty,u}= \big| \|x\| \big|_{\infty,u}$ for $x\in X$,
and
$$|A|_{\infty,u}=\max_{1\leq i\leq N}\frac{(Au)_i}{u_i}.$$

\noindent Moreover, if $Au \leq r u$ for some $r>0$, then $|A|_{\infty,u}\leq r$.

\end{proposition}
\begin{proof} The first inequality follows as in \eqref{lip}.
Note that $|u|_{\infty,u}=1$ and $|A|_{\infty,u} \geq |A u|_{\infty,u}=\mu$
where
$$
\mu =\max_{1\leq i\leq N}\frac{(Au)_i}{u_i}.
$$
On the other hand, for any $z\in \R^N$,  note that $|Az| \leq A|z| \leq  Au |z|_{\infty,u}$,
and, consequently,
$$
|Az|_{\infty,u}\leq |z|_{\infty,u} \mu,
$$
i.e. $|A|_{\infty,u}\leq \mu$. Hence
\begin{equation}\label{norm-equality}
|A|_{\infty,u} = \mu.
\end{equation}
\indent Now note that, for any $x,y\in K$,
$$
\|T(x)-T(y)\| \leq A \|x-y\|\leq A u \|x-y\|_{\infty,u},
$$
which additionally yields $\|T(x)-T(y)\|_{\infty,u}\leq \mu \|x-y\|_{\infty,u}$.

\vspace{2mm}

If $A u \leq r u$, then by applying \eqref{norm-equality} we get
$$
|A|_{\infty,u} = \max_{1\leq i\leq N} \frac{(Au)_i}{u_i}\leq r,
$$
which ends the proof.
\end{proof}

\begin{remark}
Weighted $\ell^1$- and $\ell^\infty$-type product norms are convenient for scalarization in products of Banach spaces. A product of at least two spaces is not strictly convex in either of these norms, even if all factors have uniformly convex norms. This prevents a direct appeal to the uniform convexity of the renormed product.  It does not prevent the use of normal structure: finite weighted $\ell^\infty$-products of spaces with normal structure retain normal structure by Belluce, Kirk and Steiner~\cite[Theorem~2.1]{Belluce}. This is the geometric fact used in the next section where the case $r(A)=1$ is considered. If $A$ is irreducible and all spaces $X_i$ are uniformly convex, then so is $X$ with the $\ell^2$-norm $\|\cdot\|_M$.
\end{remark}

\section{The critical case $r(A)=1$}\label{recursive}
Let $X=\prod_{i=1}^N X_i$, where the $X_i$ are Banach spaces, and let $K\subset X$ be nonempty and closed. We consider a mapping $T:K\to K$ satisfying the matrix Lipschitz estimate
$$\|T(x)-T(y)\|\leq A\|x-y\|, \ \ \mbox{ for all } \ \ x,y\in K,
$$
where $A\in\R_+^{N\times N}$. The preceding results naturally raise the question whether an analogous fixed point theory holds true in the critical case $r(A)=1$.
The strict-contraction renorming argument no longer applies when $r(A)=1$. Since the spectral
radius does not exceed any subordinate matrix norm,  no equivalent norm can make the associated matrix a strict contraction and we can not expect Banach's contraction principle to apply.

The purpose of the present Section is to show that there are two essentially different mechanisms leading to fixed point theorems in the critical case. The first one is based on an appropriately adapted weighted lattice norm. If there exists a vector $u\gg0$ such that
$Au\leq u$, then Proposition~\ref{irr1} provides a product norm under
which the mapping becomes nonexpansive. One may therefore apply classical fixed point theory for nonexpansive mappings.

The second mechanism is algebraic. If a positive subinvariant vector is not available, then the Frobenius normal form allows to decompose the problem into a sequence of lower-dimensional fixed point problems. Together, these
two approaches cover all nonnegative matrices whose spectral radius is equal to one.

\medskip

\begin{theorem}\label{critical-general-K}
Let $X_i$, $i=1,\ldots,N$, be Banach spaces having normal structure and put
$$
X=\prod_{i=1}^N X_i.
$$
Let $K\subset X$ be a nonempty, weakly compact and convex set and let
$T:K\to K$ satisfy
$$
\|T(x)-T(y)\|\leq A\|x-y\|, \ \ \mbox{ for all } \ \ x,y\in K,
$$
where $A\in\R^{N\times N}$ is a nonnegative matrix. Assume that there exists $u\in\R^N$
\begin{equation}\label{subinv-pos-vec}
u \gg 0  \ \ \mbox{ such that} \ \  Au\leq u.
\end{equation}
Then $T$ has a fixed point.
\end{theorem}

\begin{proof}
Equip $X$ with the weighted supremum norm
$$
\|x\|_{\infty,u} =  \left| \|x\|\right|_{\infty,u} = \max_{1\leq i\leq N}\frac{\|x_i\|_i}{u_i}.
$$
Since $Au\leq u$, in view of Proposition \ref{irr1}, we get
$|A|_{\infty,u}\leq 1$ and
$$
\|T(x)-T(y)\|_{\infty,u} \leq \|x-y\|_{\infty,u}, \ \ \mbox{ for all } \ \ x,y\in K.
$$
Thus $T$ is nonexpansive.

Rescaling the norm $\|\cdot\|_i$ of $X_i$ by $u_i^{-1}$ preserves normal structure, and the
finite $\ell^\infty$-product of spaces with normal structure has normal
structure; see Belluce, Kirk and Steiner \cite[Theorem~2.1]{Belluce}. Hence $K$ has normal structure with respect to
$\|\cdot\|_{\infty,u}$. The existence of a fixed point follows from
Theorem~\ref{BGK} (ii).
\end{proof}

\begin{remark} (i) The weighted supremum norm in this proof is not, in general, uniformly convex,
even if all $X_i$, $i=1,\ldots, N$ are so. Uniform convexity of the coordinate spaces implies reflexivity and normal structure; the Belluce--Kirk--Steiner theorem then gives normal structure for the weighted product $\ell^\infty$-norm. Note also that $\ell^1$-type norm need not preserve normal structure (see \cite{Landes}).

(ii) In view of Theorem \ref{irr} (ii), if $Au\leq u$ and $A^Tw\leq w$ for some $u,w\gg 0$, then $T$ is nonexpansive with respect to $\|\cdot\|_M$, where $M=\mathrm{diag}(w_1/u_1,...,w_N/u_N)$. The normal structure of $\|\cdot\|_M$ is guaranteed by Landes~\cite{Landes}.

\end{remark}
\begin{corollary}\label{nonexp-arbitrary-K} Assume that $X$, $K$ and $T$ satisfy all the assumptions of
Theorem~\ref{critical-general-K} except \eqref{subinv-pos-vec}. If $A$ is irreducible and $r(A)=1$, then $T$ has a fixed point.
\end{corollary}
\begin{proof}
In view of the Perron--Frobenius theorem -- see Theorem \ref{Per-Frob}, there exists
$u\gg0$ such that $Au=u$.
The conclusion follows immediately from
Theorem~\ref{critical-general-K}.
\end{proof}

For a reducible matrix, the sufficient condition $Au\leq u$ with $u\gg0$ may fail, even when $r(A)=1$. For example, take
$$
A=\begin{bmatrix}1&c\\0&1\end{bmatrix},
$$
with $c>0$. Here $Au\leq u$ is impossible for $u\gg0$. Nevertheless, this matrix is may be used by the following theorem. Thus the matrix estimate need not possess nonexpansive scalarizations.

The key observation is that reducibility itself provides additional
structure. Indeed, after a suitable simultaneous permutation of rows
and columns, every reducible matrix assumes \emph{Frobenius normal form} -- see Theorem \ref{fnf}.

\begin{theorem}\label{critical-main}
Let $X_i$, $i=1, \ldots, N$, be Banach spaces with normal structure and
$$
X=\prod_{i=1}^N X_i.
$$
Let $K\subset X$ be nonempty, weakly compact and convex. Assume that $T:K\to K$ satisfies the condition
$$\|T(x)-T(y)\|\le A\|x-y\|, \ \ \mbox{ for all } \ \ x,y\in K,
$$
where a nonnegative matrix $A\in\R^{N\times N}$ satisfies $r(A) \le 1 $. Then $T$ possesses a fixed point.
\end{theorem}
\begin{proof} Since $K$ is weakly compact, it is norm-closed and hence complete. If $r(A)<1$, then the conclusion follows from Perov's theorem.
Hence we may assume that $r(A)=1$. If $A$ is irreducible, Corollary~\ref{nonexp-arbitrary-K} concludes the proof. We therefore continue with the reducible case.
\medskip

\noindent
{\bf Step 1. Reduction to Frobenius normal form:}
After a simultaneous permutation of rows and columns we may assume that $A$ is in Frobenius normal form
$$A=\begin{bmatrix}
A_{11}&A_{12}&\cdots&A_{1m}\\
0&A_{22}&\cdots&A_{2m}\\
\vdots&\ddots&\ddots&\vdots\\
0&\cdots&0&A_{mm}
\end{bmatrix},
$$
where each diagonal block is either irreducible or the zero
$1\times1$ matrix, and $r(A)=\max_{i=1,...,m}r(A_{ii})$ -- see Theorem \ref{fnf}.

Since the permutation induces a linear homeomorphism of the
corresponding product spaces and conjugates $T$ to a mapping
satisfying the matrix Lipschitz estimate with the permuted matrix $P^{-1}AP$, see \eqref{perm}. Therefore without loss of generality we may assume that $A$ already has the above block form.

\medskip

\noindent {\bf Step 2. Construction of the reduced problem:} Let $n_j$ denote the size of the block $A_{jj}$ and define
$$
s_0:=0,\qquad s_j:=n_1+\cdots+n_j,\qquad j=1,\ldots,m.
$$
If we put
$$
Y_j:=\prod_{i=s_{j-1}+1}^{s_j}X_i,\qquad j=1,\ldots,m,
$$
then we clearly have
$$
X=Y_1\times\cdots\times Y_m.
$$
For $j=1,\ldots,m$ let $\pi_j\colon X\to Y_j$ be the canonical projection onto the
$j$-th block, and let $S_j:=\pi_j\circ T\colon K\to Y_j$, i.e.
$$
S_j(x)=\bigl(T_{s_{j-1}+1}(x),\ldots,T_{s_j}(x)\bigr),\quad x\in K.
$$
Then $T=(S_1,\ldots,S_m)$.

\medskip

We claim that for each diagonal block, the corresponding component $S_j$ depends only on the variables belonging to the blocks numbered by $j,j+1,\ldots,m$. Indeed, if we fix $j\in\{1,\ldots,m\}$ and suppose that
$x,y\in K$ coincide in all blocks $j,j+1,\ldots,m$, then the corresponding coordinates of $\|x-y\|$ vanish.  Since $A$ is block upper triangular,
the same is true for the $j$-th block of $A\|x-y\|$. Consequently,
$\|S_j(x)-S_j(y)\|=0$, and therefore $S_j(x)=S_j(y)$.

\medskip

Now let $Y=Y_2\times\cdots\times Y_m$ and $q:=(\pi_2,\ldots,\pi_m)\colon X\to Y$.
By the claim above, the formula
$$
S(q(x)):=q(T(x))=\bigl(S_2(x),\ldots,S_m(x)\bigr),\qquad x\in K,
$$
defines a mapping $S\colon q(K)\to q(K)$.

\medskip

\noindent{\bf Step 3. Solving the reduced problem:} We now proceed by induction on the number $m$ of diagonal
Frobenius blocks.

\medskip

If $m=1$, then $A$ is irreducible or is the $1\times1$ zero matrix. If $r(A)<1$, including the zero-matrix case, Perov's theorem applies; if $r(A)=1$, Corollary~\ref{nonexp-arbitrary-K} applies. We use induction for the full assertion with $r(A)\leq1$, since the reduced problem may have spectral radius strictly smaller than one.

Assume now that the assertion has been proved whenever the Frobenius
normal form consists of at most $m-1$ diagonal blocks.

Let
$$
A'=\begin{bmatrix}
A_{22}&\cdots&A_{2m}\\
\vdots&\ddots&\vdots\\
0&\cdots&A_{mm}
\end{bmatrix}.
$$
Since
$$
r(A')=\max_{2\le j\le m}r(A_{jj})\le1,
$$
the induction hypothesis applies to $S:q(K)\to q(K)$. Indeed, $q$ is weakly continuous and linear, so $q(K)$ is nonempty, weakly compact and convex. For any $y,y'\in q(K)$, choose lifts $x,x'\in K$. The zero blocks in the first block column below $A_{11}$ show that
$$
\|S(y)-S(y')\|\leq A'\|y-y'\|,
$$
independently of the first coordinates of the chosen lifts. The remaining coordinate spaces still have normal structure. Hence there exists $y_2^*\in q(K)$ such that $S(y_2^*)=y_2^*$.

\medskip

\noindent {\bf Step 4. The invariant fibre:} Having determined a fixed point of the reduced problem, we now recover
the first block. Define
$$
L:=\{z\in Y_1\mid (z,y_2^*)\in K\}.
$$
Since $y_2^* \in q(K)$, the fibre $L$ is nonempty.

Moreover,
$$
L\times\{y_2^*\}=K\cap(Y_1\times\{y_2^*\}).
$$
The affine subspace $Y_1\times\{y_2^*\}$ is   closed, and its identification with $Y_1$ is a weakly continuous linear homeomorphism. Consequently, $L$ is weakly compact and convex, and in particular is norm-closed and complete.

Define $F:L\to Y_1$ by
$$
F(z):=S_1(z,y_2^*),  \ \ \mbox{ for all  } z\in L.
$$
Since $S(y_2^*)=y_2^*$ and, for each $z\in L$, we have
$$
(F(z),y_2^*)= (S_1(z,y_2^*), y_2^*) =(S_1(z,y_2^*), S(y_2^*))=T(z, y_2^*) \in K,
$$
we get $F(z)\in L$. Therefore $F(L)\subset L$.

Moreover,
\begin{equation}\label{olip}
\|F(z)-F(z')\|\leq A_{11}\|z-z'\|, \qquad z,z'\in L.
\end{equation}
The block $A_{11}$ is either irreducible or the zero $1\times1$ matrix.
If $A_{11}=[0]$, then \eqref{olip} implies that $F$ is constant; since $F(L)\subset L$, its constant value is a fixed point of $F$. Otherwise, if $r(A_{11})<1$, Perov's theorem applies. If $r(A_{11})=1$, then Corollary~\ref{nonexp-arbitrary-K} applies to $F$ on $L$. Consequently, there exists $y_1^*\in L$ such that $F(y_1^*)=y_1^*$. Hence
$$
T(y_1^*,y_2^*)= (S_1(y_1^*,y_2^*), S (y_2^*)) = (F(y_1^*), S(y_2^*)) = (y_1^*, y_2^*),
$$
which completes the induction and the whole proof.
\end{proof}

\begin{corollary}\label{uniformly-convex-critical}
Assume, in addition to the matrix Lipschitz condition of Theorem~\ref{critical-main}, that every space$X_i$, $i=1,...,N$, is uniformly convex and $K$ is bounded, closed and convex. If $r(A)\leq 1$, then $T$ has a fixed point.
\end{corollary}
\begin{proof} Every uniformly convex Banach space is reflexive and has normal structure. Hence each $X_i$ is reflexive and has normal structure, and therefore their finite product is reflexive. Since $K$ is  bounded, closed and convex, it is weakly compact. Theorem~\ref{critical-main} now applies.
\end{proof}

The spectral threshold $r(A)\leq1$ cannot be enlarged in general. Indeed, for $r(A)>1$ fixed points may fail to exist even in Hilbert spaces, as in the following example that is  appropriately adapted from the classical construction of Kakutani~\cite{Kakutani}.
\begin{example}
\label{beals-product}
Let $ X =X_1\times X_2$, where $X_1:=\R^n$ with $\|\cdot\|_1:=|\cdot|_2$ and $X_2:=\ell^2$ with $\|\cdot\|_2:=\|\cdot\|_{\ell^2}$. Set
$$
K:= \left\{
x=(x_1,x_2)\in X:
\|x_1\|_1^2+\|x_2\|_2^2\leq1
\right\}.
$$
Thus $K$ is the closed unit ball of the Hilbert space
$X_1\oplus_2X_2$. In particular, $K$ is nonempty, weakly compact
and convex, while both $X_1$ and $X_2$ are uniformly convex.

Fix $0<\eps\leq1$, and let $e_1=(1,0,\ldots,0)\in\R^n$.
We define $T:K\to X$, $T=(T_1,T_2)$, by
$$
T_1(x)
=
\eps\left(1-\bigl|\|x\|\bigr|_2\right)e_1,\; x\in K,$$
where $\|x\|=(\|x_1\|_1,\|x_2\|_2)\in\R^2$ and $|\cdot|_2$ stands for the Euclidean norm in $\R^2$, and
$$T_2(x)=J(x),$$
where $J:X\to\ell^2$ denotes the natural linear isometry given by
$$
J(x)=J(x_1,x_2)
=
(x_{1,1},\ldots,x_{1,n},x_{2,1},x_{2,2},\ldots).
$$
Let $x\in K$ and put $ t:=\bigl|\|x\|\bigr|_2\leq 1$. Since $J$ is an isometric embedding,
$$\Bigl|\bigl\|T(x)\bigr\|\Bigr|_2^2=\|T_1(x)\|_1^2+\|T_2(x)\|_2^2
=\eps^2(1-t)^2+t^2\leq(1-t)^2+t^2\leq1.
$$
Hence $T(K)\subset K$.

By the reverse triangle inequality
\begin{align*}
\|T_1(x)-T_1(y)\|_1
& =
\eps \left|
\bigl| \|x\| \bigr|_2-\bigl|\|y\|\bigr|_2
\right|
\leq
\eps
\left(
\|x_1-y_1\|_1^2+\|x_2-y_2\|_2^2
\right)^{1/2}\\
&\leq
\eps\|x_1-y_1\|_1+\eps\|x_2-y_2\|_2.
\end{align*}
Moreover,
$$
\begin{aligned}
\|T_2(x)-T_2(y)\|_2
&=
\|J(x_1-y_1,x_2-y_2)\|_2
=
\left(
\|x_1-y_1\|_1^2+\|x_2-y_2\|_2^2
\right)^{1/2}\\
&\leq
\|x_1-y_1\|_1+\|x_2-y_2\|_2.
\end{aligned}
$$
Consequently,
$$
\|T(x)-T(y)\|
\leq
A\|x-y\|,
\qquad x,y\in K,
$$
where
$$
A=
\begin{bmatrix}
\eps&\eps\\
1&1
\end{bmatrix}.
$$
The eigenvalues of $A$ are $0$ and $1+\eps$, and therefore $r(A)=1+\eps>1$.

Suppose that $T(x)=x$ for some $x=(x_1,x_2)\in K$. The second
component of the fixed point equation gives $x_2=J(x_1,x_2)$, which implies that $$
x_{2,n+k}=x_{2,k}, \ \ \mbox{ for all }\ \  k\ge 1,
$$
and therefore $x_2$ is periodic with period $n$. A nonzero periodic sequence cannot belong to $\ell^2$.
Hence $x_2=0$, and therefore also $x_1=0$. Thus $0=T_1(0)=\eps e_1$. This is a contradiction showing that $T$ has no fixed points.
\end{example}

\section{Applications}\label{applications}

The results presented above appear particularly well suited to coupled systems of equations in which the strengths of the interactions between the individual subsystems are not uniformly distributed.
\subsection{A Dirichlet boundary value problem}
We consider systems of the second-order Dirichlet boundary value problems
\begin{equation}\label{system-of-bvps}
\begin{cases}
-\ddot u_i (t)=f_i (t,u_1(t), \ldots, u_N(t)), & t\in[a,b], \ i= 1,\dots, N,\\
\ \ u_i(a)=u_i(b)=0, & i= 1,\dots, N\end{cases}
\end{equation}
with $f_i:[a,b]\times E_1\times  \ldots \times E_N\to E_i$ where
$(E_i,|\cdot|_i)$, $i=1,...,N$, are Banach spaces. We assume that $f_i$, $i=1,\ldots, N$  are Carath\'eodory maps: $f_i(\cdot,x)$ is measurable for all $x\in E_1\times \ldots \times E_N$ and $f_i(t,\cdot)$ is continuous for almost all $t\in [a,b]$. We look for its weak solutions $u_i:[a,b]\to E_i$ such that $u_i\in H^2((a,b),E_i)\cap H_0^1((a,b),E_i)$ for $i=1,\ldots, N$.

Problems of this type have been investigated by
numerous authors in the case $N=1$,
frequently by means of Sadovskii's or Darbo's fixed point
theorem; see, for example, \cite{Chandra}.
In particular, Zhou and Peng~\cite{Zhou} (see also the
extensive bibliography therein) consider the
Dirichlet boundary value problem
$-\ddot u(t)=f(t,u(t))$, $u(0)=u(1)=0$,
where $f:[0,1]\times E\to E$ is a continuous
mapping and $E$ is a Banach space.
Authors assume the linear
growth condition $\|f(t,x)\|_E\leq c_0+c_1\|x\|_E$, where
$c_1<\pi^2$, together with an estimate involving the Kuratowski
measure of noncompactness
$$\alpha(f([0,1]\times D))\leq L\alpha(D), \mbox{ for any bounded } D\subset E,
$$
with $L<4$. Alternatively, results in \cite{Zhou} are obtained under
Lipschitz-type assumptions of the form
$$\|f(t,x)-f(t,y)\|_E\leq L\|x-y\|_E,\ \  t\in [0,1],\; x,y\in E
$$
for the Lipschitz constant $L<\pi^2$.

\medskip

\indent Here we suppose that $f_i$, $i=1,\ldots, N$, satisfy a matrix Lipschitz condition, for each $i=1,...,N$:
\begin{equation}\label{ml}|f_i(t,x_1,\dots, x_N)-f_i(t,y_1,\dots, y_N)|_i\leq \sum_{j=1}^NL_{ij}|x_j-y_j|_j,
\end{equation}
for any $t\in [a,b]$, $x_j,y_j\in E_j$,  $j=1,\ldots, N$,  where $L_{ij}\geq 0$, and $|f_i(\cdot,0)|_i\in L^2[a,b]$, $i=1,...,N$.
Observe that \eqref{ml} implies that for any $1\leq i\leq N$, $f_i$ satisfies a growth condition:
\begin{equation}\label{gc}|f_i(t,x)|_i\leq |f_i(t,0)|_i+\sum_{j=1}^NL_{ij}|x_j|_j.\end{equation}
\indent If we consider $f=(f_1, \ldots, f_N):[a,b]\times E\to E$ where
$E:=E_1 \times E_2 \times \cdots \times E_N$, then \eqref{system-of-bvps} can be rewritten as
\begin{equation}\label{bvp}
\begin{cases}
-\ddot u(t)=f(t,u(t)), & t\in[a,b],\\
\ \ u(a)=u(b)=0,\end{cases}
\end{equation}
and look for its strong solutions \(u:[a,b]\to E\) in the space $u\in H^2((a,b),E)\cap H_0^1((a,b),E)$.\\
 \indent A standard fixed point formulation of \eqref{bvp} is obtained by means of the
Green function $g$ associated with the differential operator $-d^2/dt^2$ and the
Dirichlet boundary conditions. Namely \eqref{bvp} is equivalent to
the Hammerstein-type fixed-point problem
$$
u=GN_f(u),\;\; u\in X:=L^2([a,b],E)=\prod_{i=1}^NL^2([a,b],E_i),
$$
where $N_f:X\to X$ is the superposition operator determined by $f$, i.e.
$$N_f(u):=f(\cdot,u(\cdot)),\; u\in X,$$
and $G:X\to X$ is a linear Hammerstein integral operator given by
$$G(u)=\int_a^b g(\cdot,s)u(s)\,ds=\left(\int_a^bg(\cdot,s)u_1(s)\,ds,\dots,
\int_a^b g(\cdot,s)u_N(s)\,ds\right),\; u\in X,$$
where
\begin{equation}\label{green-function}
g(t,s)=\frac{1}{b-a}
\begin{cases}
(t-a)(b-s), & a\le t\le s\le b,\\[1mm]
(s-a)(b-t), & a\le s\le t\le b.
\end{cases}
\end{equation}
Observe that, in view of \eqref{ml}, $N_f$ is well-defined, $g(t,s)\geq 0$ on $[a,b]^2$ and $G$ maps $X$ into $\prod_{i=1}^NH^2((a,b),E_i)\cap H_0^1((a,b),E_i)$.\\[0.3em]
\indent The Dirichlet Green operator
$(-d^2/dt^2)^{-1}$ acting on $L^2(a,b)$ is compact,
positive, and self-adjoint. Hence, by the spectral theorem,
its norm equals $\lambda_1^{-1}$, where
$$\lambda_1=\frac{\pi^2}{(b-a)^2}$$ is the first Dirichlet eigenvalue of $-d^2/dt^2$;
see \cite[\S 5.4]{Teschl}. The same norm equality holds for the Green operator
acting on each Bochner space $X_i:=L^2([a,b],E_i)$.
Indeed, if we denote it still by $G$ and since the Green kernel $g(t,s)$ is nonnegative,
we have
$$|(Gu)(t)|_i\leq \int_a^b g(t,s)|u(s)|_i\,ds,$$
which implies
$$\mu_i:=\|G\|_{\mathcal L(X_i)}\leq\lambda_1^{-1}=\frac{(b-a)^2}{\pi^2}.$$
The reverse inequality follows by considering functions
of the form $u(t)=\varphi_1(t)x$, where $\varphi_1$
is a first Dirichlet eigenfunction and $x\in E_i$
is an arbitrary unit vector.

\indent Let $T:X\to X$ be given by $T=G\circ N_f$, i.e.
\begin{equation}\label{hammer}
(Tu)(t) = GN_f(u) (t)=
\int_a^b g(t,s)f(s,u(s))\,ds,\; u\in X.
\end{equation}
It is then clear that
$$T=(T_1,\ldots,T_N)\;\; \text{with}\;\; T_i(u)=\int_a^b g(t,s)f_i\big(s,u(s)\big)\,ds,\; u\in X,\, i=1,\ldots,N,$$
and that (strong) solutions of \eqref{bvp} are precisely the fixed points of $T$. For $i=1,...,N$, we endow the space $X_i$ with the usual norm $\|u\|_i:=\sqrt{\int_a^b|u(t)|_i^2\,dt}$, for $u\in X_i$.
Then
$$
\|T_i(u)-T_i(v)\|_i\leq \sum_{j=1}^N a_{ij}\|u_j-v_j\|_j,  \ \ \mbox{ for any } \ \ u,v \in X,
$$
where
$$a_{ij}=\mu_iL_{ij}=L_{ij}(b-a)^2/\pi^2.$$

In view of Perov's theorem we have
\begin{proposition}
If the spectral radius $r(A)< 1$, i.e. when
\begin{equation}\label{warunek1} r(L)<\frac{\pi^2}{(b-a)^2},\end{equation}
for $L=[L_{ij}]$, then system \eqref{system-of-bvps} has a unique solution.
\end{proposition}
\indent If condition \eqref{warunek1} is violated, then the Perov fixed-point argument no longer applies, and solutions may fail to exist. In general if $r(L)\geq \pi^2/(b-a)^2$, then neither existence nor uniqueness is guaranteed by the Lipschitz condition alone. For instance in the scalar case when $E=\R$ and $N=1$, and $L=L_{11}=\frac{\pi^2}{(b-a)^2}$, then, in view of the Fredholm alternative, the problem
$$-\ddot u=Lu+h(t),\; u(a)=u(b)=0$$
with $h\in L^2[a,b]$, i.e., when $f(t,x)=h(t)+Lx$, $t\in [a,b]$, $x\in\R$, has no solution if $h$ is not orthogonal to the first eigenfunction $\varphi_1(t)=\sin((b-a)^{-1}\pi(t-a))$ of $-\ddot u$ with Dirichlet boundary conditions.

\medskip

At this place it is worth recalling that normal structure is
preserved by Bochner $L^p$-spaces for
$1<p<\infty$. More precisely, by a theorem of
Smith and Turett~\cite{Smith}, the space
$L^p([a,b],Y)$, where $Y$ is a Banach space, has normal structure if and only
if $Y$ has normal structure. Furthermore, $L^p([a,b],Y)$ is reflexive (resp. uniformly convex) if and only if $Y$ is reflexive (resp. uniformly convex).

Consequently, if $E_i$ is reflexive
and has normal structure (or is uniformly convex), then
$X_i=L^2([a,b],E_i)$ is reflexive with
normal structure.
\begin{theorem} Suppose that each $E_i$ is reflexive with normal structure. If, for each $i=1,...,N$,
$$|f_i(t,x)|_i\leq m_i(t),$$
where $m_i\in L^2[a,b]$ and
$$r(L)=\frac{\pi^2}{(b-a)^2},$$
then system \eqref{system-of-bvps} has a solution.
\end{theorem}
\begin{proof} For $i=1,...,N$, let $R_i=(b-a)^2\pi^{-2}\|m_i\|_{L^2[a,b]}$ and let
$$K=\{u\in X\mid \|u_i\|_i\leq R_i,\;i=1,...,N\}.$$
Then $K$ is convex and closed, hence weakly compact ($X$ is reflexive) and $T(X)\subset K$. Thus Theorem  \ref{critical-main} yields a solution.
\end{proof}

\begin{remark}\label{Volterra} The above fixed point approach should be contrasted with those corresponding to initial value problems. Suppose that $f$ satisfies the matrix Lipschitz condition \eqref{ml}. For
\[
\dot u(t)=f(t,u(t)),\;u(a)=u_0\in E,
\]
the associated integral operator has the Volterra form
\begin{equation}\label{triang}
T(u)(t)=u_0+\int_a^t f(s,u(s))\,ds,\; u\in X.
\end{equation}
Then $T=(T_1,...,T_N)$, where
\begin{equation}\label{triang1}T_i(u)=u_{0i}+\int_a^t f_i(s,u(s))\,ds,\; u\in X,\end{equation}
and
$$\|T_i(u)-T_i(v)\|_i\leq\sum_{j=1}^N(b-a)L_{ij}\|u_j-v_j\|_j,\; u,v\in X.$$
Take the norm $|\cdot|_\infty$ in $\R^N$. Then, for all $u,v\in X:=C([a,b],E)$,
$$\|T(u)-T(v)\|_\infty\leq (b-a)\lambda\|u-v\|_\infty,\; \text{where}\; \lambda:=|L|_\infty.$$
The Volterra form of integrals in \eqref{triang} and \eqref{triang1} allows us to use the  Bielecki modification (see \cite{Biel, Cord}) of the norm $\|\cdot\|_\infty$ in $X=C([a,b],E)$, i.e., for $u\in X$,
\begin{align*}\|u\|_\infty&=\Big|\big(\sup_{t\in [a,b]}|u_1(t)|_1,...,\sup_{t\in [a,b]}|u_N(t)|_N\big)\Big|_\infty\\&=\max_{1\leq i\leq N}\Big(\sup_{t\in[a,b]}|u_i(t)|_i\Big)=\sup_{t\in [a,b]}\Big(\max_{1\leq i\leq N}|u_i(t)|_i\Big).\end{align*}
Namely we consider
$$\|u\|_\infty^B:=\sup_{t\in [a,b]}e^{-\lambda(t-a)}\max_{1\leq i\leq N}|u_i(t)|_i,\; u\in X.$$
Then, for any $u,v\in X$,
$$\|T(u)-T(v)\|_\infty^B\leq \big(1-e^{-\lambda(b-a)}\big)\|u-v\|_\infty^B,$$
i.e. $T$ is a Banach contraction and in order to infer the existence of solutions no condition such as \eqref{warunek1} on the spectral radius $r(L)$ is needed.
\end{remark}

In Volterra-type problems, Bielecki renorming often helsp to  convert the associated operator directly into a Banach contraction, making additional spectral-radius assumptions unnecessary. From this viewpoint the need to apply the Perov or Browder--G\"ohde--Kirk theorem to  Volterra-type problems is rather limited.

In contrast, the Hammerstein operator $T$ associated with the boundary-value problem \eqref{bvp} is of Fredholm type: for every fixed $t$, integration is performed over the whole interval $[a,b]$. The causal triangular structure characteristic of Volterra operators is therefore absent, and the elementary one-sided exponential Bielecki renorming does not yield an arbitrarily small Lipschitz constant in the same automatic way; comp. \cite{Bess}.

\subsection{An application to periodic problems}
\label{sec:periodic}
Consider the following periodic problem
\begin{equation} \label{system-of+eqs:periodic-ode}
\dot u_i (t)=f_i(t,u_1(t), \ldots, u_N(t)),\ \ t\in [0,T], \ \ i = 1,\ldots, N,
\end{equation}
with $f_i:\mathbb R\times X_1\times \dots \times X_N \longrightarrow X_i$, $i=1,\dots, N$, where
$(X_i,\|\cdot\|_i)$ are reflexive Banach spaces, $i=1,\dots, N$.  Assume also that
$f_i$, $i=1, \ldots, N$, is continuous, has sublinear growth, i.e., for all $t\in\R$ and $x\in X$,
$$\|f_i(t,x_1,\dots,x_N)\|_i\leq a_i(t)+b_i\sup_{1\leq j\leq N}\|x_j\|_j,$$
where $a_i\in L^1_{loc}(\R)$ and $b_i\geq 0$,
and is $T$-periodic with respect to the first variable,
$$
f_i(t+T,x_1, \dots, x_N)=f_i(t,x_1, \ldots,x_N),\;\; t\in\mathbb R,\quad x_j \in X_j, \quad i, j =1,\ldots, N.
$$
Moreover we assume that the so-called {\em matrix one-sided Lipschitz estimate} is satisfied:  for all $i,j=1,...,N$ there exist functions
$$\beta_{ij}\in L^1[0,T],\;\;  \beta_{ij}(t)\geq 0\;\;\text{if}\; i\neq j,$$
such that, for any $i=1,...,N$, almost every $t\in[0,T]$ and all $\xi_j,\eta_j\in X_j$, $j=1,\ldots, N$,
\begin{equation}
\label{eq:diss-h}
\la \xi_i-\eta_i, f_i(t,\xi_1, \ldots, \xi_N)-f_i(t,\eta_1, \ldots, \eta_N) \ra_i \leq \sum_{j=1}^N\beta_{ij}(t)\|\xi_i-\eta_i\|_i\|\xi_j-\eta_j\|_j,
\end{equation}
where $\la\cdot,\cdot\ra_i$ stands for the (lower) semi-inner product associated to the norm $\|\cdot\|_i$ in $X_i$ -- see Section \ref{lumer}. This condition expresses dissipativity in the $i$-th component (observe that $\beta_{ii}$ is not necessarily nonnegative) and the coupling between different components.\\[0.3em]
\indent Let us now indicate how the matrix comparison principle can be used
to study periodic solutions of systems of ordinary differential
equations in product Banach spaces. Our approach is inspired by Browder~\cite{Browder1}.

Let $$X=\prod_{i=1}^N X_i,$$ and rewrite \eqref{system-of+eqs:periodic-ode} as
\begin{equation}
\label{eq:periodic-ode}
\dot u(t)=f(t,u(t)),
\end{equation}
where $f=(f_1, \dots, f_N):\mathbb R\times X\longrightarrow X$. Under our assumptions,  in view of Deimling~\cite[\S3.2,\,\S3.3 and \S3.4]{Dei}, for any $x\in X$, the Cauchy problem
\begin{equation}\label{cp}\dot u(t)=f(t,u(t)),\;\; u(0)=x,\end{equation}
admits a unique $C^1$-solution  $u(\cdot\,;x):[0,T]\longrightarrow X$. The associated Poincar\'e map  $P:X\to X$ is defined by
$$P(x):=u(T;x),\;\; x\in X.$$
By the $T$-periodicity of $f$, a solution $u(\cdot;x)$ extends to
a $T$-periodic solution of \eqref{eq:periodic-ode} if and only if $P(x)=x$.
Thus the periodic problem is reduced to a fixed-point problem for
the Poincar\'e map.\\[0.3em]
\indent  Let $x,y\in X$ and set
$$u(t)=u(t;x),\;\; v(t)=u(t;y),  \ \ t\in [0,T].
$$
For $i=1,\ldots,N$ define
$$ p_i(t):=\|u_i(t)-v_i(t)\|_i\;\; \text{and}\;\; p(t)=(p_1(t),...,p_N(t)),\;\; t\in [0,T].$$
Fix $i=1,...,N$. Since $u_i-v_i$ is differentiable and $\|\cdot\|_i$ is Lipschitz, we gather that $p_i$ is absolutely continuous and differentiable almost everywhere. According to Lemma \ref{diff} (iii), the left derivative $D_-p_i(t)$ exists for almost all $t$ and
\begin{equation}\label{pochodna}
p_i(t)D_-p_i(t)=\la u_i(t)-v_i(t),\dot u_i (t)- \dot v_i (t)\ra_i.
\end{equation}
In view of \eqref{eq:diss-h} we obtain
$$
D_- p_i(t)\leq \sum_{j=1}^N\beta_{ij}(t)p_j(t).
$$
In order to get it either we divide by $p_i(t)>0$ or if $p_i(t) = 0$, then, for almost every $t\in [0,T]$,
$$
D_- p_i (t) = [0, \dot u_i(t) - \dot v_i (t) ]_i = - \| \dot u_i (t) - \dot v_i(t)\|_i \leq 0 \leq
\sum_{j\neq i} \beta_{ij}(t) p_j(t)
$$
because $B(t)=[\beta_{ij}(t)]_{i,j=1}^N$ is a Metzler matrix (all of its off-diagonal entries are nonnegative). Therefore
\begin{equation}
\label{eq:matrix-dini-periodic}
D_-p(t)\le B(t)p(t)
\end{equation}
for almost all $t\in [0,T]$.\\[0.3em]

\indent Let $E_1$ be the full-measure set on which the left derivative exists and \eqref{eq:matrix-dini-periodic} holds, and let $E_2$ be the full-measure set on which $p$ is differentiable in the usual sense. For $t\in E_1\cap E_2,$ the ordinary derivative exists, and therefore necessarily
$D_-p(t)=p'(t).$ This implies that for almost all $t$
$$\dot p(t)\leq B(t)p(t).$$
In view of the comparison principle (see Section \ref{comparison} and Proposition \ref{compa})
\begin{equation}
\label{eq:p-comparison-periodic}
p(t)
\le
V(t,0)p(0),
\qquad
0\le t\le T,
\end{equation}
where  $V(t,s)$ is the fundamental matrix of
$$\frac{\partial}{\partial t}V(t,s) = B(t)V(t,s),\;\; V(s,s)=I.$$
Since $p(0)=\|x-y\|$ and $p(T)=\|u(T)-v(T)\|=\|u(T;x)-u(T;y)\|=\|P(x)-P(y)\|$ we get
\begin{equation}\label{zasa}
 \|P(x)-P(y)\|\leq A\|x-y\|, \ \ \mbox{ for all } \ \  x,y\in X,
\end{equation}
where $A:=V(T,0)$
is the monodromy matrix corresponding to $B$; $A\geq 0$ by Proposition \ref{compa}.  These observations allow one to apply directly the fixed-point results developed in the preceding sections.

\begin{theorem} If
\begin{enumerate}
\item[\em (i)] $r(A)<1$, or
\item[\em (ii)] $r(A)=1$, each space $X_i$, $i=1,...,N$, has normal structure, and there exists a nonempty weakly compact convex set $K\subset X$ such that $P(K)\subset K$,
\end{enumerate}
then the equation \eqref{eq:periodic-ode} has a $T$-periodic orbit, which is unique in case (i).
\end{theorem}
\begin{proof}
The first case is covered by Perov's theorem, while the second one follows from Theorem \ref{critical-main}.
\end{proof}
\begin{remark}
The assumption concerning the invariant set $K$ is natural.
In applications it is usually obtained from an \emph{a priori} estimate,
a dissipativity condition, or an invariant-region argument.
\end{remark}
\begin{example}
Suppose each $X_i$, $i=1,...,N$, is reflexive, has normal structure and that there is a continuously Fr\'echet differentiable,
convex, and coercive functional $g:X\to\mathbb R$ such that
\[
Dg(u)f(t,u)\leq0
\]
for every $u\in X$ and $t\in[0,T]$. Choose
$c>\inf_X g$ and put
\[
K_c:=\{x\in X:g(x)\leq c\}.
\]
Then $K_c$ is nonempty, bounded closed and convex. Hence $K_c$ is weakly compact since $X$ is reflexive. If $x\in K_c$, the chain rule
gives
\[
\begin{aligned}
g(P(x))-g(x)
 &=g(u(T;x))-g(u(0;x))=\int_0^T Dg(u(s;x))\dot u(s;x)\,ds\\
 &=\int_0^T
   Dg(u(s;x))f(s,u(s;x))\,ds
 \leq0.
\end{aligned}
\]
It follows that $g(P(x))\leq g(x)\leq c$, and hence $P(K_c)\subset K_c$.
\end{example}
\begin{remark} If $B(t)\equiv B$ is constant, then $A=e^{TB}$.
Consequently,
$$
r(A)=r(e^{TB})=e^{T s(B)},
$$
where
$$s(B)=\max\{\operatorname{Re}\lambda:\lambda\in\sigma(B)\}$$
is the spectral bound of $B$. Hence
$$r(A)<1\;\Longleftrightarrow\;s(B)<0.$$
\indent The critical situation $r(A)=1$ corresponds to $s(B)=0.$ This illustrates clearly the role played by the comparison system: the existence theory for periodic solutions of the nonlinear equation is reduced to spectral information about a finite-dimensional
nonnegative monodromy matrix.
\end{remark}

\begin{example}
Consider
\begin{equation}
\label{eq:periodic-two-components}
\begin{cases}
\dot u_1=-\alpha_1(t)u_1+g_1(t,u_1,u_2),\\[1mm]
\dot u_2=-\alpha_2(t)u_2+g_2(t,u_1,u_2),
\end{cases}
\end{equation}
where for $i=1,2$, $\alpha_i:\R\to[0,\infty)$ is continuous and $T$-periodic, $g_i:\R\times X_1\times X_2\to X_i$, where $X_i$ is a Banach space, is continuous, $T$-periodic, satisfies  sublinear growth condition and, for any $\xi =(\xi_1,\x_2),\eta(\eta_1,\eta_2) \in X_1\times X_2$,
$$\|g_1(t,\xi)-g_1(t,\eta)\|_1\leq\beta_{11}(t)\|\xi_1-\eta_1\|_1+\beta_{12}(t)\|\xi_2-\eta_2\|_2,
$$
$$\|g_2(t,\xi)-g_2(t,\eta)\|_2\leq\beta_{21}(t)\|\xi_1-\eta_1\|_1+
\beta_{22}(t)\|\xi_2-\eta_2\|_2,$$
where $\beta_{ij}\in L^1[0,T]$, $\beta_{ij}\geq 0$ a.e. Property analogous to \eqref{eq:diss-h} for $f_i(t,\xi_1,\xi_2)=-\alpha_i(t)\xi_i+g_i(t,\xi_1,\xi_2)$, $i=1,2$, follows from Lemma \ref{diff} (i) and (ii).  Indeed, for $\xi=(\xi_1,\xi_2)$, $\eta=(\eta_1,\eta_2)$ and $\{i,j\}=\{1,2\}$, applying Lemma~\ref{diff}~(ii) with $x=\xi_i-\eta_i$, $\lambda=-\alpha_i(t)$ and $y=g_i(t,\xi)-g_i(t,\eta)$, and then Lemma~\ref{diff}~(i) together with the above estimate of $g_i$, we get, for almost every $t\in[0,T]$,
\begin{align*}
\la \xi_i-\eta_i,f_i(t,\xi)-f_i(t,\eta)\ra_i
&=-\alpha_i(t)\|\xi_i-\eta_i\|_i^2+\la \xi_i-\eta_i,g_i(t,\xi)-g_i(t,\eta)\ra_i\\
&\leq\bigl(-\alpha_i(t)+\beta_{ii}(t)\bigr)\|\xi_i-\eta_i\|_i^2+\beta_{ij}(t)\|\xi_i-\eta_i\|_i\|\xi_j-\eta_j\|_j.
\end{align*}%
Then the comparison matrix is
$$B(t)=\begin{bmatrix}
-\alpha_1(t)+\beta_{11}(t) & \beta_{12}(t)
\\[1mm]
\beta_{21}(t) & -\alpha_2(t)+\beta_{22}(t)
\end{bmatrix}.
$$
Thus, the existence  of a $T$-periodic solution can be
studied through the monodromy matrix $V(T,0)$ of the  positive linear comparison system
$$\dot z(t)=B(t)z(t).$$
This formulation makes explicit the competition between the
componentwise dissipation represented by the functions $\alpha_i$
and the coupling represented by the functions $\beta_{ij}$.
\end{example}

\appendix

\renewcommand{\thesection}{\Alph{section}}

\makeatletter
\def\@seccntformat#1{%
  \@ifundefined{#1@cntformat}%
    {\csname the#1\endcsname.\hspace{0.5em}}
    {\csname #1@cntformat\endcsname}
}
\newcommand{\section@cntformat}{}   
\makeatother
\section{Appendix}

This appendix collects auxiliary material used in the main text. We first recall formulas for subordinate matrix norms, collect several standard facts concerning the spectral radius and Perron--Frobenius theory, and summarize the Frobenius normal form for reducible matrices. We then state the comparison principle and recall relevant facts concerning semi-inner products in Banach spaces.

\subsection{Classical subordinate matrix norms}

We use the order, lattice norms and subordinate matrix norms introduced in Section~\ref{basic-notation}.
For the classical norms one has
$$|A|_1=\max_{1\leq j\leq N}\sum_{i=1}^N|a_{ij}|,\qquad |A|_\infty=\max_{1\leq i\leq N}\sum_{j=1}^N|a_{ij}|,
$$
while
$$|A|_2=\sqrt{\sup\sigma(A^TA)},$$
where $\sigma(A^TA)$ is the spectrum of $A^TA$ and $A^T$ is the transpose of $A$.

\subsection{Spectral radius}

The Gelfand formula states that
$$
r(A)
=
\lim_{n\to\infty}|A^n|_*^{1/n},
$$
where $|\cdot|_*$ is any subordinate matrix norm on $\R^N$. In
particular,
\begin{equation}\label{subo1}
r(A)\le |A|_*
\end{equation}
for every subordinate matrix norm.

\begin{remark}\label{gel}
The inequality \eqref{subo1} may be strict, even by a large margin. Thus the condition $r(A)<1$
does not imply that $|A|_*<1$ with respect to a given subordinate
matrix norm. This observation explains the role of the renorming
techniques developed in Section~\ref{strict-case}, where suitable equivalent lattice
norms are constructed in order to obtain the estimate
$|A|_*<1$.

For example, if
$$
A=\begin{bmatrix}
a&c\\
0&b
\end{bmatrix},
$$
then $r(A)=\max\{|a|,|b|\}$,  whereas
$$
|A|_1=\max\{|a|,|b|+|c|\},
\qquad
|A|_\infty=\max\{|a|+|c|,|b|\}.
$$
Hence, choosing $|c|$ sufficiently large, one obtains matrices for
which $r(A)$ is arbitrarily small while $|A|_p$, $1\leq p\leq\infty$, is arbitrarily large.
\end{remark}

\begin{theorem}\label{r(A)<1-app}
For a matrix $A\in\R^{N\times N}$ the following conditions are
equivalent.

\begin{enumerate}
\item[\rm(a)]
$r(A)<1$;

\item[\rm(b)]
$A^n\to0$ as $n\to\infty$ in any matrix norm;

\item[\rm(c)]
$A^nx\to0$ for every $x\in\R^N$;

\item[\rm(d)]
the Neumann series
$$
\sum_{k=0}^\infty A^k
$$
converges and
$$
(I-A)^{-1}
=
\sum_{k=0}^\infty A^k.
$$
\end{enumerate}
\end{theorem}
The equivalence of the above conditions is classical and follows, for
instance, from the Gelfand formula; see
Horn and Johnson~\cite{Horn} or Bhatia~\cite{Bhatia} or Rheinboldt and Vandergraft~\cite[Lemma 1]{Van}.

\subsection{Perron--Frobenius theory}

Nonnegative matrices possess remarkable spectral properties discovered
by Perron and Frobenius. We recall only those facts that are used in the
present paper. Standard references include
\cite{Horn, Varga, Marek}.

\begin{theorem}[Perron--Frobenius] \label{Per-Frob}
Let $A\ge0$.
\begin{enumerate}
\item[\rm(a)] The spectral radius $r(A)$ is an eigenvalue of both
$A$ and $A^T$ with nonnegative corresponding eigenvectors.
\item[\rm(b)] If $A$ is irreducible, then
\begin{enumerate}
\item[\rm(i)]
$r(A)>0$, except for the $1\times1$ zero matrix;
\item[\rm(ii)]
$r(A)$ is a simple eigenvalue of
$A$ and $A^T$;
\item[\rm(iii)]
there exist unique, up to positive scalar multiples,
right and left Perron eigenvectors
$u,w\gg0$, satisfying
$$
Au=r(A)u, \qquad A^Tw=r(A)w;
$$
\item[\rm(iv)]
every eigenvalue satisfying
$|\lambda|=r(A)$
is simple.
\end{enumerate}
\item[\rm(c)]
If $A^k\gg0$ for some integer $k\geq1$, then
$r(A)$ is the unique eigenvalue satisfying
$|\lambda|=r(A)$.
\end{enumerate}
\end{theorem}
For proofs and further information see Horn and Johnson~\cite{Horn}, Varga~\cite{Varga},
or Marek~\cite{Marek}.
\medskip

Recall that a permutation matrix $P$ is determined by a permutation $p$ through $Pe_i=e_{p(i)}$; it is clear that $P^{-1}=P^T$.  With the definition in Section~\ref{basic-notation}, $A$ is reducible if and only if there exists a permutation matrix $P$ such that
$$
P^{-1}AP= \begin{bmatrix}
A_1&B\\
0&A_2
\end{bmatrix},
$$
where $A_1$ and $A_2$ are square matrices of positive orders.
If one of the diagonal blocks $A_1$ or $A_2$ is still reducible,
the similar decomposition can be applied to that block.
Repeating the procedure finitely many times yields the following canonical block decomposition.

\begin{theorem}[Frobenius normal form]\label{fnf}
\label{appendix-frobenius}
Let $A\in\R^{N\times N}$. Then there exists a permutation matrix
$P$ such that
$$
P^{-1}AP=\begin{bmatrix}
A_{11}&A_{12}&\cdots&A_{1m}\\
0&A_{22}&\cdots&A_{2m}\\
\vdots&\ddots&\ddots&\vdots\\
0&\cdots&0&A_{mm}\end{bmatrix},
$$
where every diagonal block $A_{ii}$ is either irreducible or a
$1\times1$ zero matrix. Moreover,
$$
\sigma(A)=\bigcup_{i=1}^m\sigma(A_{ii}),\qquad r(A)=
\max_{1\le i\le m}r(A_{ii}).
$$
If, in addition, $B=P^{-1}AP=[b_{ij}]$, then
\begin{equation}\label{perm}
b_{ij}=a_{p(i)p(j)},
\qquad i,j=1,\ldots,N,\end{equation}
where a permutation $p\in S_N$ determines $P$.
\end{theorem}
The Frobenius normal form is obtained by simultaneously permuting the
rows and columns of $A$. The diagonal blocks correspond to the
strongly connected components of the directed graph associated with
$A$, ordered topologically. This decomposition is the basis of the
recursive argument used in Section~\ref{recursive}.

\subsection{Comparison principle}\label{comparison}
\begin{proposition}[Comparison principle]\label{compa}
Let $B\in L^1(0,T;\mathbb R^{N\times N})$
be a Metzler matrix field, that is, $b_{ij}(t)\geq 0$
for $i\neq j$ for almost every $t\in[0,T]$. Let $V:\{(t,s)\mid 0\leq s\leq t\leq T\}\to\R^{N\times N}$ denote the evolution matrix associated with the problem
$$\dot z(t)=B(t)z(t).$$
Then $V(t,s)\geq 0,\; 0\leq s\leq t\leq T$. If $p:[0,T]\to \R^N$ is absolutely continuous and satisfies
$$ \dot p(t)\leq B(t)p(t)$$
componentwise for almost every $t\in[0,T]$, then
$$p(t)\leq V(t,0)p(0),\;\;t\in[0,T].$$
\end{proposition}
\noindent The result stated above is a cornerstone of the theory of positive and cooperative linear systems. The fact that a Metzler matrix field $B$ generates a nonnegative evolution matrix $V$ is the linear, time-varying version of the well-known property that the matrix exponential of a Metzler matrix is nonnegative. The subsequent comparison inequalities for the differential inequalities are classical results from the theory of differential inequalities for quasimonotone (or cooperative) systems and may be found in many textbooks (see e.g. \cite{Laks1}). In order to make the paper self-contained we propose a short proof.
\begin{proof} We shall refer to \cite[Theorem 1, p.270]{Walter}. It says (in our context) that if $u:[0,T]\to\R^N$ is absolutely continuous, $u(0)\geq 0$, $\dot u(t)\geq B(t)u(t)$ almost everywhere, then $u(t)\geq 0, \ \ t\in [0,T]$.\\
\indent First let $x\geq 0$ and put $u(t)=V(t,s)x$ for $T\geq t\geq s\geq 0$. Then
$$\dot u(t)=B(t)u(t)\;\; \text{and}\;\;  u(s)=x\geq 0.$$
Walter's Theorem 1 gives $V(t,s)x=u(t)\geq 0$. Taking $x=e_j$, $j=1,\ldots,N$, shows that every column of $V(t,s)$ is nonnegative. Hence $V(t,s)\geq 0$.\\[0.3em]
\indent Now let $y(t):=V(t,0)p(0)$ and set $u(t)=y(t)-p(t)$. Then $u(0)=0$, and
$$
\dot u(t)=B(t)y(t)-\dot p(t)\geq B(t)y(t)-B(t)p(t)=B(t)u(t).
$$
Walter's theorem yields $u(t)\ge0$, i.e.
$p(t)\le V(t,0)p(0)$ as desired. \end{proof}

\subsection{Lumer semi-inner products}\label{lumer}
Let $X$ be a Banach space with norm $\|\cdot\|$.
For $x,y\in X$, define the \emph{(left) Kato bracket}
\begin{equation}
[x,y]
:=
\lim_{h\to 0^-}
\frac{\|x+hy\|-\|x\|}{h},
\end{equation}
and the associated (lower) \emph{semi-inner product in the sense of Lumer}
\begin{equation}\label{si}
\langle x,y\rangle
:=
\|x\|\,[x,y] .
\end{equation}
See \cite[Section 1.2]{Laks2} or \cite[\S3.1]{Dei} for the general theory.
If $X$ is a Hilbert space, then the semi-inner product thus defined coincides with the standard scalar product in $X$.\\[0.3em]
\indent The bracket $[\cdot,\cdot]$ measures the one-sided directional derivative of the norm $\|\cdot\|$, and it will be used in the sequel to express dissipativity and monotonicity conditions for operators acting in $X$.\\[0.3em]
\indent We record some basic properties of $[\cdot,\cdot]$ -- see \cite[Lem. 1.2.2, Cor. 1.2.1]{Laks2} (see also \cite[Lemma 3.2]{Dei}).
\begin{lemma}\label{diff} (i) For $x,y\in X$, $\la x,y\ra\leq \|x\|\|y\|$.

(ii) For $x,y\in X$ and $\lambda\in\R$, $[x,\lambda x+y]=\lambda\|x\|+[x,y]$; consequently $\la x,\lambda x+y\ra=\lambda\|x\|^2+\la x,y\ra$.

(iii) If a function $x:(a,b]\to X$ is left differentiable at $t\in (a,b]$, then so is $n:=\|x(\cdot)\|$
and $\dot n_-(t)=[x(t),\dot x_-(t)]$. Therefore, in view of \eqref{si},
$$n(t)\, \dot n_-(t)=\la x(t),\dot x_-(t)\ra.$$

\end{lemma}



\end{document}